\documentclass[11pt]{amsart}

\usepackage{verbatim}
\usepackage{eucal,url,amssymb,stmaryrd,
enumerate,amscd,
}

\usepackage[pagebackref,colorlinks=true]{hyperref}  

\usepackage{amsfonts}
\usepackage{amsmath,amsthm,amssymb,amscd,enumerate,eucal,url,stmaryrd}

\numberwithin{equation}{section}

\newtheorem{thrm}{Theorem}[section]

\newtheorem{lemma}[thrm]{Lemma}

\newtheorem{prop}[thrm]{Proposition}

\newtheorem{cor}[thrm]{Corollary}

\newtheorem{rmrk}[thrm]{Remark}

\def\Ob{{\nabla^{ob}}}
\def\Ro{{R^{ob}}}

\begin{document}

\begin{abstract}
We show that on an HKT manifold the holonomy of the Obata
connection is contained in $SL(n,\mathbb H)$ if and only if the
Lee form is an exact one form. As an application, we show compact
HKT manifolds with holomorphically trivial canonical bundle which
are not balanced. A simple criterion for non-existence of HKT metric
on hypercomplex manifold is given in terms of the Ricci-type tensors of the Obata connection.
\end{abstract}

\title[HKT manifolds with holonomy $SL(n,H)$]
{HKT manifolds with holonomy $SL(n,H)$}
\date{\today}

\author{Stefan Ivanov}
\address[Ivanov]{University of Sofia ``St. Kl. Ohridski"\\
Faculty of Mathematics and Informatics\\
Blvd. James Bourchier 5\\
1164 Sofia, Bulgaria} \email{ivanovsp@fmi.uni-sofia.bg}

\author{Alexander Petkov}
\address[Petkov]{University of Sofia ``St. Kl. Ohridski"\\
Faculty of Mathematics and Informatics\\
Blvd. James Bourchier 5\\
1164 Sofia, Bulgaria} \email{a\_petkov\_fmi@abv.bg}

\maketitle

\setcounter{tocdepth}{2} \tableofcontents

\section{Introduction}
We recall that an HKT structure on an hyperhermitian manifold is a
linear connection  with totally skew-symmetric torsion preserving
the hyperhermitian structure. If the torsion three form is closed
(resp. trace-free) the HKT structure is called strong (resp.
balanced). If the torsion vanishes one has hyperK\"ahler manifold.
HKT structures are present in many branches of theoretical and
mathematical physics. For instance, they appear on supersymmetric
sigma models with Wess-Zumino term \cite{GHR,HP,HP1} as well as in
supergravity theories \cite{Str,GPS,PT}.

There are known some geometrical and topological properties of HKT
manifold.  A simple characterization of the existence of HKT
structure is obtained in terms of the intrinsic torsion of an
$Sp(n)Sp(1)$ structure \cite{CS}. It is shown in \cite{MS,BS}
that, as in the hyperK\"ahler case, locally any HKT metric admits
an HKT potential. A version of Hodge theory has been given in
\cite{V3} discovering the remarkable analogy between the de Rham
complex of a K\"ahler manifold and the Dolbeault complex on HKT
manifold.

A special attention is paid for HKT manifold with holomorphically
trivial canonical bundle with respect to any complex structure
from the  hypercomplex family. HKT manifold with holomorphic
volume form appear as solutions of the gravitino and dilatino
Killing spinor equations in dimension $4n$ with more then two
supersymmetries preserved \cite{Str}. It was observed by
M.Verbitsky in \cite{V5} that, in the compact case,  the latter
condition  can be expressed in terms of the Obata connection
\cite{Ob} which is  the unique torsion-free connection preserving
the hypercomplex structure. M. Verbitsky proved in \cite{V5} that
a compact HKT manifold has holomorphically trivial canonical
bundle exactly when the holonomy of the Obata connection is a
subgroup of the special quaternionic linear group $SL(n,\mathbb
H)$ which leads  to invent into consideration the notion of
$SL(n,\mathbb H)$ manifolds. The group $SL(n,\mathbb H)$ is one of
possible holonomy groups of a torsion-free linear connection in
the Merkulov-Schwachh\"ofer list \cite{MS}. The $SL(n,\mathbb H)$
manifolds were studied in \cite{AV2,V7} discovering that the
quaternionic Dolbeault complex can be identified with a part of
the de Rham complex. For a hypercomplex manifold with
holomorphically trivial canonical bundle admitting an HKT metric,
a version of Hodge theory constructed in \cite{V3} leads to the
fact established in \cite{V5} that a compact hypercomplex manifold
with holomorphically trivial canonical bundle is $SL(n,\mathbb H)$
manifold if it admits an HKT-structure.  A. Swann constructed in
\cite{S}  compact simply connected $SL(n,\mathbb H)$ manifolds
which do not admit any HKT structure which, in particular, shows the
existence of compact hypercomplex manifolds with holomorphically
trivial canonical bundle with no HKT metric.

Special attention deserve balanced HKT metrics. It was shown in
\cite{V08} that a balanced HKT manifold is an $SL(n,\mathbb H)$
manifold. Balanced HKT metrics seem to be the quaternionic
analogue of the Calabi-Yau metrics defined in terms of
quaternionic Monge-Ampere equation \cite{AV2,V08}. A quaternionic
version of the famous Calabi-Yau theorem conjectured in
\cite{AV2,V08}  states that on a compact HKT manifold with
$SL(n,\mathbb H)$-holonomy of the Obata connection there exists a
balanced HKT metric and if it exists it is unique in its
cohomology class.

The main purpose of this note is to find  precise  simple
condition on an HKT manifold to have holomorphically trivial
canonical bundle. We show in Theorem~\ref{main2} below that the
necessary and sufficient condition an HKT manifold to have
$SL(n,\mathbb H)$-holonomy of the Obata connection is that a
certain trace of the torsion three form, called the Lee form, is
an exact one form.

Examples of compact HKT manifold with holomorphically trivial
canonical bundle are all nilmanifolds with an abelian hypercomplex
structure since they are balanced HKT spaces \cite{BDV}. Compact
examples of such spaces which are not nilmanifold were presented
in \cite{BF} which are again balanced. There are known compact
simply connected  HKT manifolds with holomorphically trivial
canonical bundle constructed by A. Swann \cite{S} via the twist
construction.

\begin{rmrk}
Applying Theorem~\ref{main2} to the explicit examples of HKT
manifold presented  in [\cite{BF}, Example 6.1 and Example 6.2] we
obtain that these HKT manifold are $SL(2,\mathbb H)$-manifold
since the corresponding Lee form is exact. This provides
non-balanced compact HKT manifolds with holomorphically trivial
canonical bundle.
\end{rmrk}

\medskip
\noindent {\bf Acknowledgments.} We would like to thank A. Swann
for useful comments and remarks. We also thank the referee for remarks making the exposition  more clear and understandable. A.P. is partially supported by
the Contract 181/2011 with the University of Sofia
`St.Kl.Ohridski'. S.I. is partially supported by
the Contract 181/2011 with the University of Sofia
`St.Kl.Ohridski' and Contracts ``Idei", DO 02-257/18.12.2008 and
DID 02-39/21.12.2009.

\section{Hypercomplex, $SL(n,\mathbb H)$ and HKT manifolds}
An almost hypercomplex structure on a 4n-dimensional manifold $M$
is a triple $H=(J_s), s=1,2,3$, of almost complex structures
$J_s:TM\rightarrow TM$ satisfying the quaternionic identities
$J_s^2=-id_{TM}$ and $J_1J_2=-J_2J_1=J_3$. When each $J_s$ is a
complex structure, $H$ is said to be a hypercomplex structure on
$M$  and $(M,H)$ is called a hypercomplex manifold. A
hyperhermitian metric is a Riemannian metric $g$ which is
Hermitian with respect to each almost complex structure in $H$,
$g(J_s.,J_s.)=g(.,.), s=1,2,3$. A hyperhermitian manifold
$(M,H,g)$  consists of a hypercomplex structure $H$ and a
compatible hyperhermitian metric $g$. The fundamental 2-forms of a
hyperhermitian manifold $(M,H,g)$ are globally defined by
$
F_s(.,.)=g(.,J_s.),\quad s=1,2,3.$ 
If the three
fundamental 2-forms are closed we have \emph{hyperK\"ahler
manifold}.
\subsection{HKT manifolds}
A hyperhermitian manifold $(M,H,g)$ is called \emph{hyperk\"ahler
with torsion (HKT-manifold)} if there exists a linear connection
preserving the hyperhermitian structure and having totally
skew-symmetric torsion, or equivalently, if the following
condition is satisfied \cite{HP}
\begin{equation}\label{hkt}
\begin{aligned}
&J_1dF_1=J_2dF_2=J_3dF_3,\quad {\rm where}\\
&J_sdF_s(X,Y,Z)=-dF_s(J_sX,J_sY,J_sZ), \quad s=1,2,3.
\end{aligned}
\end{equation}
Each $F_s$ is an (1,1)-form with respect to $J_s$. One can also
associate a non-degenerate complex 2-form
$\Omega_i=F_j+\sqrt{-1}F_k, $
where $\{i,j,k\}$  is
a cyclic permutation of $\{1,2,3\}$. The 2-form $\Omega_i$ is of
type (2,0) with respect to the complex structure $J_i$. If the
2-form $\Omega_i$ is closed then the manifold is hyperK\"ahler.
The hyperk\"ahler with torsion condition is equivalent to the
condition $\partial_{J_i}\Omega=0$ \cite{GP}. It was observed in
\cite{CS} that the condition \eqref{hkt} implies that the almost
hypercomplex structure is hypercomlex, thus reducing the
definition of HKT-manifold as an almost hyperhermitian manifold
satisfying \eqref{hkt}.

\subsection{Supersymmetry and HKT manifolds}
The notion of HKT-manifold was  introduced  in physics by Howe and
Papadopoulos \cite{HP} in connection with (4,0) supersymmetric
sigma models with non vanishing Wess-Zumino term. HKT manifolds
are also connected with the supersymmetric string backgrounds \cite{Str}.
The bosonic fields of the
ten-dimensional supergravity which arises as low energy effective
theory of the heterotic string are the spacetime metric $g$, the
NS three-form field strength $H$, the dilaton $\phi$ and the gauge
connection $A$ with curvature $F^A$. One considers the connection
$\nabla=\nabla^g +\frac12 H$,
where $\nabla^g$ is
the Levi-Civita connection of the Riemannian metric $g$. The
connection $\nabla$ preserves the metric, $\nabla g=0$ and has
totally skew-symmetric torsion $T=H$.

A heterotic geometry will preserve  supersymmetry if and only if
there exists at least one Majorana-Weyl spinor $\epsilon$ such
that the supersymmetry variations of the fermionic fields vanish,
i.e. the following Killing-spinor equations hold \cite{Str}
\begin{equation}\label{sup1}
\delta_{\lambda}=\nabla\epsilon=0; \qquad
\delta_{\Psi}=(d\phi-\frac12H)\cdot\epsilon=0; \qquad
\delta_{\xi}=F^A\cdot\epsilon=0,
\end{equation}
where  $\lambda, \Psi, \xi$ are the gravitino, the dilatino and the
gaugino  fields, respectively and $\cdot$ means Clifford action of
forms on spinors.

The whole Strominger system has an
additional equation, called anomaly cancellation (see \cite{Str}),
expressing $dH$ in terms of the first Pontryagin form of the
instanton connection $A$ and a certain connection on the tangent
bundle (which turns out  to be of instanton type \cite{I2}). We
note that the first compact solutions  to the whole Strominger
system with non-trivial fields in dimension six and non-constant
dilaton were constructed in \cite{y1} (see also \cite{y2,y3,y4})
and the first explicit compact examples with constant dilaton are
constructed in \cite{FIUV}.

We briefly explain here the geometry arising from the first two
equations in \eqref{sup1}. The first equation in \eqref{sup1} leads, in even dimensions $2n$,
to  the existence of an $SU(n)$-structure, i.e. the existence of
an almost complex structure $J$ hermitian compatible with the
metric $g$ and a non-vanishing complex (with respect to $J$)
volume form  which are preserved by the metric connection with
totally skew-symmetric torsion $\nabla$, hence the holonomy group
of $\nabla$ is contained in $SU(n)$. The second equation in
\eqref{sup1} forces the almost complex structure to be integrable
and the non-vanishing complex volume form to be a holomorphic
volume form, i.e. complex manifold with holomorphically trivial
canonical bundle \cite{Str}. It turns out that in  the case of
compact non-K\"ahler solution to the first two equations in
\eqref{sup1} the holomorphic (n,0) form is unique \cite{AI,IP1}.
Strominger  shows \cite{Str} that, in the complex case, the
torsion three form of $\nabla$ is unique given by
\begin{equation}\label{torsu}
T=JdF_J,
\end{equation}
where $F_J$ is the fundamental two form of $(g,J)$. The equation
\eqref{torsu} combined with the classical result that a metric
connection is  completely determined by its torsion  implies that
the connection $\nabla$ preserving the hermitian structure $(J,g)$
and having totally skew-symmetric torsion always exists and it is
unique. Thus one has the notion of K\"ahler manifold with torsion
(KT manifold), $(M,J,g,\nabla)$. It follows from \eqref{torsu}
that the torsion three form is of type (1,2)+(2,1) since the
almost complex structure is integrable. The connection $\nabla$
with torsion three form preserving a hermitian structure was
independently used by Bismut \cite{Bis} to prove a local index
formula for the Dolbeault operator when the manifold is not
K\"ahler and it was sometimes called the Bismut connection.

If the torsion three form is closed, $dT=0$ then the KT-manifold
is called \emph{strong KT manifold}. These spaces are connected
with the supersymmetric string background of type IIA, IIB (see eg
\cite{GMW} and references therein). It is easy to see from
\eqref{torsu} that the strong KT-condition $dT=0$ is equivalent to
the condition $\partial\bar{\partial}F_J=0$. Such hermitian spaces
are also known as pluriclosed hermitian manifolds and a Ricci-type
flow is investigated in \cite{ST}. If the trace of the exterior
derivative of the torsion is zero, $g(dT,F_J)=0$, (equivalently the trace of  $\partial\bar{\partial}F_J$ is zero, $g(\partial\bar{\partial}F_J,F_J)=0$),  we have the notion of almost strong KT manifold and vanishing
theorems on compact almost strong KT manifolds are presented in
\cite{IP,IP1}.

Note that for almost hermitian manifold the  existence of a
connection with totally skew-symmetric torsion preserving the
almost hermitian structure is obstructive. It was observed in
\cite{FI} that the obstruction is encoded into the properties of
the Nijenhuis tensor $N_J$, namely, such a connection exists on an
almost hermitian manifold exactly when the Nijenhuis tensor is a
three form, $N_J(X,Y,Z)=g(N_J(X,Y),Z)=-g(N_J(X,Z),Y)$. In this
case the connection is unique and its torsion three form $T$ is
given by \cite{FI}
\begin{equation}\label{torsua}
T=JdF_J + N_J.
\end{equation}
On an almost Hermitian manifold the (3,0)+(0,3) part $dF^-_J$  of
the exterior derivative of the fundamental 2-form is  determined
by the Nijenhuis tensor \cite{Gau} and if the Nijenhuis tensor is
a three form then the formula takes the form \cite{FI}
$JdF_J^-=-\frac34N_J.$  An important special case is the Nearly
K\"ahler manifold \cite{Gr} which is characterized by the
condition $JdF_J=JdF_J^-=-\frac34N_J$.

In the case of dimension $4n$, the existence of more than  two
parallel spinors in the first equation of \eqref{sup1} leads  to
the existence of an $Sp(n)$-structure, i.e. the existence of an
almost hyperhermitian structure $(g,H)$ which is preserved by the
metric connection with totally skew-symmetric torsion $\nabla$,
hence the holonomy group of $\nabla$ is contained in $Sp(n)$ (see also the recent paper \cite{HPS}).
Since for each $J_s\in H$ the connection $\nabla$ is unique, we
obtain from \eqref{torsua}
\begin{prop}\label{torpr}
An almost hyperhermitian manifold $(M,g,H)$ admits a connection
$\nabla$ with  skew-symmetric torsion preserving the
hyperhermitian structure if and only if the Nijenuis tensors
$N_{J_1},N_{J_2},N_{J_3}$ are three forms and the following
conditions hold
\begin{equation*}
J_1dF_{J_1} + N_{J_1}=J_2dF_{J_2} + N_{J_2}=J_3dF_{J_3} + N_{J_3}.
\end{equation*}
\end{prop}
If the hypercomplex structure is integrable we have an HKT
manifold. However,  an almost hyperhermitian structure which is
consisted of three Nearly K\"ahler structures admits
hyperhermitian connection with torsion three form if and only if
it is hyper K\"ahler. Indeed, the nearly K\"ahler conditions
$J_sdF_{J_s}=-\frac34N_{J_s}, s=1,2,3$ together with
Proposition~\ref{torpr} imply
$J_1dF_{J_1}=J_2dF_{J_2}=J_3dF_{J_3}$ and the already mentioned
result in \cite{CS} gives that the hypercomplex structure is
integrable and we have hyper K\"ahler manifold.

The second equation in \eqref{sup1} forces that each  almost
complex structure $J_s\in H$ is  integrable \cite{Str}, i.e. we
have a hyperhermitian  manifold which has to satisfy \eqref{hkt}.
Thus, we have an HKT manifold with torsion three form given by
\begin{equation}\label{3torh}
T=J_1dF_{J_1}=J_2dF_{J_2}=J_3dF_{J_3}.
\end{equation}
Following \cite{Str}, the first two  equations in \eqref{sup1}
imply that the HKT manifold admits a non-degenerate holomorphic
(2n,0) form with respect to any complex structure $J_s\in H$, i.e.
the canonical bundle of the corresponding KT manifolds
$(M,g,J_s\in H,\nabla)$ is holomorphically trivial. It was
observed by M. Verbitsky in \cite{V5} that, in the compact case,
the latter condition may be expressed in terms of the Obata
connection of a hypercomplex manifold leading to invent into
consideration the notion of $SL(n,\mathbb H)$ manifolds.

\subsection{$SL(n,\mathbb H)$ manifolds}
It was shown by Obata in \cite{Ob} that a hypercomplex manifold
$(M,H)$ admits a unique torsion-free connection preserving the
complex structures $J_s, s=1,2,3$. We shall call this connection
\emph{the Obata connection} and denote it with $\Ob$. The Obata
connection is uniquely determined by the conditions
$\Ob J_1=\Ob J_2=\Ob J_3=T^{ob}=0$.
The converse is also true, namely if an almost hypercomplex
manifold admits a torsion-free  connection preserving the almost
hypercomplex structure then it is a hypercomplex manifold.

The holonomy of the Obata connection,  $Hol(\Ob)$ is contained in
the general quaternionic linear group $GL(n,\mathbb H)$ since
$\Ob$ preserves the hypercomlex structure. The holonomy of the
Obata connection is one of the most important invariants on an
hypercomplex manifold and it is rarely known explicitly except
when an hyperK\"ahler metric exists and the Obata connection
coincides with the Levi-Civita connection whose holonomy group is
contained in $Sp(n)$. It was shown very recently in \cite{Sol}
that the holonomy of the Obata connection on the Lie group $SU(3)$
coincides with $GL(2,\mathbb H)$. An important subgroup inside
$GL(n,\mathbb H)$ is its commutator $SL(n,\mathbb H)$ which
appears in the Merkulov-Schwachh\"ofer list \cite{MS} of possible
holonomy groups of a torsion-free linear connection. This group
can be defined as a group of quaternionic matrices preserving a
non-zero complex valued form $\Phi\in \Lambda_{\mathbb
C}^{2n,0}({\mathbb H}^n_{J_1})$, where $\mathbb H^n_{J_1}$ is the
$n$-dimensional quaternionic vector space $\mathbb H^n$ considered
as a $2n$-dimensional complex vector space with respect to the
complex structure $J_1$. A hypercomplex manifold with holonomy of
the Obata connection inside $SL(n,\mathbb H)$ is called
\emph{$SL(n,\mathbb H)$-manifold}. It was observed by Verbitsky in
\cite{V5} that $(M,J\in H)$ has holomorphically trivial canonical
bundle for any $SL(n,\mathbb H)$ manifolds $(M,H)$. For any
$SL(n,\mathbb H)$-manifold $(M,H)$ and any complex structure $J\in
H$ there is a holomorphic  volume form $\Phi\in
\Lambda^{2n,0}(M,J)$ with respect to $J$ which is parallel with
respect to the Obata connection \cite{V5,BDV}.

For a hypercomplex manifold with holomorphically trivial canonical
bundle admitting an  HKT metric, a version of Hodge theory was
constructed in \cite{V3} which leads to the fact established in
\cite{V5} that a compact hypercomplex manifold with
holomorphically trivial canonical bundle is an $SL(n,\mathbb H)$
manifold if it admits an HKT-structure. Compact simply connected hypercomplex
manifold with holomorphically trivial canonical bundle were constructed by
A. Swann \cite{S} where it is also shown that some of these examples do not admit any HKT structure.

\subsection{The Lee form of an HKT manifold}

We recall that the Lee form $\theta$ of an almost hermitian structure $(g,J)$ is defined by
$\theta = \delta F_J\circ J,$
where $\delta$ is the co-differential. A hermitian manifold  with
vanishing Lee form is called \emph{balanced} \cite{Mi}. For a
KT-manifold the Lee form can be expressed in terms of the torsion
as follows \cite{IP1}
\begin{equation}\label{leet}
\theta(X) = -\frac12\sum_{i=1}^{2n}T(JX,e_i,Je_i),
\end{equation}
where $e_1,\dots,e_{2n}$ is an orthonormal basis.

For an HKT manifold, it follows from \eqref{3torh} that  the torsion three form of an
HKT manifold is of type (1,2)+(2,1) with respect to each complex
structure $J_s\in H$. It was shown in \cite{I1} that in such a
case one has the identities
\begin{equation}\label{leeh}
\sum_{a=1}^{4n}T(J_1X,e_a,J_1e_a)=\sum_{a=1}^{4n}T(J_2X,e_a,J_2e_a)=\sum_{a=1}^{4n}T(J_3X,e_a,J_3e_a).
\end{equation}
Here and further  $e_1,\dots,e_{4n}$ will be an orthonormal basis
of $TM$.

Combining \eqref{leeh} with the  expression of the Lee form in
terms of the torsion, \eqref{leet}, we get that the three Lee
forms on an HKT-manifold coincide \cite{IP1} thus obtaining a
globally defined one form $\theta$ on any HKT manifold defined by
\eqref{leet}, where $J\in H$ \cite{I1,IP1,I3}. We call this one
form \emph{the Lee form of the HKT manifold}. If the Lee form of
an HKT manifold vanishes then we have the notion of \emph{a
balanced HKT manifold} (see also \cite{V08}).

It was shown in \cite{V08} that a balanced HKT manifolds is an
$SL(n,\mathbb H)$ manifold but the converse is not true.

The purpose of this note is to find a necessary and sufficient
condition an HKT manifold to be an $SL(n,\mathbb H)$ manifold,
i.e. $Hol(\Ob)\subset SL(n,\mathbb H)$. We show that this happens
exactly when the Lee form is an exact form.

The aim of the paper is to prove the following
\begin{thrm}\label{main2}
On  an HKT manifold the following conditions are equivalent:
\begin{itemize}
\item[a)] The HKT manifold is an $SL(n,\mathbb H)$ manifold,
i.e. $Hol(\Ob)\subset SL(n,\mathbb H)$.
\item[b)] The Lee form is an exact form.
\end{itemize}
\end{thrm}
Combining Theorem~\ref{main2} with the already mentioned result of
Verbitsky \cite{V5}, we obtain
\begin{cor}\label{cotriv}
A compact HKT-manifold has holomorphically trivial canonical
bundle if and only if the  Lee form is  exact.
\end{cor}

\section{Proof of the main result}

First we calculate the difference between the HKT connection and
the Obata  connection. Surprisingly, we expressed the difference
only in terms of the HKT torsion. We  have
\begin{prop}\label{obhkt}
On an HKT manifold the Obata connection and the HKT connection are
related by
\begin{equation}\label{conn}
\begin{aligned}
&g(\Ob_XY,Z)=g(\nabla_XY,Z) +A(X,Y,Z),\quad {\rm where}\\
&2A(X,Y,Z)= -T(X,J_1Y,J_1Z)-T(J_1X,J_1Y,Z)-T(X,J_3Y,J_3Z)-T(J_1X,J_3Y,J_2Z).
\end{aligned}
\end{equation}
\end{prop}
\begin{proof}
Obata wrote in \cite{Ob} a formula connecting the Obata connection
with a linear   connection with torsion tensor $T$ preserving the
hypercomplex structure. Following the proof of [\cite{Ob},Theorem
10.4] and using the vanishing of the Nijenhuis tensors of a
hypercomplex structure, one finds that two connections are related
by an (1,2) tensor $B$ having the expression
\begin{multline}\label{obex}
-4B(X,Y)=T(X,Y)-J_1T(X,J_1Y)-J_2T(X,J_2Y)-J_3T(X,J_3Y)\\+
T(J_1X,J_1Y)+J_1T(J_1X,Y)-J_2T(J_1X,J_3Y)+J_3T(J_1X,J_2Y).
\end{multline}
In particular, for the HKT connection we use the special
properties of its torsion,  namely it is a three form which is of
type (1,2)+(2,1) with respect to any complex structure $J\in H$,
i.e. the next identities hold
\begin{equation}\label{12form}
\begin{aligned}
T(X,Y,Z)-T(JX,JY,Z)-T(JX,Y,JZ)-T(X,JY,JZ)=0, \qquad J\in H;\\
T(J_iX,J_iY,Z)-T(J_kX,J_kY,Z)-T(J_kX,J_iY,J_jZ)-T(J_iX,J_kY,J_jZ)=0,
\end{aligned}
\end{equation}
where $\{i,j,k\}$ is a cyclic permutation of $\{1,2,3\}$. Now,
using \eqref{12form}, we obtain easily that $B$ given in
\eqref{obex} is equal to $A$ described with the second equation in
\eqref{conn}.
\end{proof}
We need  the following important
\begin{lemma}\label{1}
On an HKT manifold the difference tensor $A$ between the Obata
connection and the HKT connection satisfies the identities
\begin{equation}\label{traces}
\sum_{a=1}^{4n}A(X,e_a,e_a)=-2\theta(X);\qquad
\sum_{a=1}^{4n}A(X,e_a,J_se_a)=0,\quad s=1,2,3.
\end{equation}
\end{lemma}
\begin{proof} We calculate from \eqref{conn}  applying \eqref{leet},
 using the quaternionic identities and  the fact that the torsion is a three form that
\begin{multline}\label{2}
\sum_{a=1}^{4n}A(X,e_a,e_a)=-\frac12\sum_{a=1}^{4n}T(J_1X,J_1e_a,e_a)-\frac12\sum_{a=1}^{4n}T(J_1X,J_3e_a,J_2e_a)\\
=-\theta(X)+\frac12\sum_{a=1}^{4n}T(J_1X,e_a,J_1e_a)=-2\theta(X).
\end{multline}
Similarly, we obtain the next sequence of identities
\begin{equation}\label{3}
\sum_{a=1}^{4n}A(X,e_a,J_1e_a)=\frac12\sum_{a=1}^{4n}\Big(T(X,J_1e_a,e_a)-T(X,J_3e_a,J_2e_a)\Big)=0;
\end{equation}
\begin{multline}\label{31}
\sum_{a=1}^{4n}A(X,e_a,J_2e_a)= -\frac12\sum_{a=1}^{4n}\Big(T(X,J_1e_a,J_1J_2e_a)+T(J_1X,J_1e_a,J_2e_a)\Big)\\
-\frac12\sum_{a=1}^{4n}\Big(T(X,J_3e_a,J_3J_2e_a)-T(J_1X,J_3e_a,e_a)\Big)
=\sum_{a=1}^{4n}\Big(T(X,e_a,J_2e_a)-T(J_1X,e_a,J_3e_a)\Big)\\=2\theta(J_2X)-2\theta(J_3J_1X)=0;\end{multline}
\begin{multline}\label{32}
\sum_{a=1}^{4n}A(X,e_a,J_3e_a)= -\frac12\sum_{a=1}^{4n}\Big(T(X,J_1e_a,J_1J_3e_a)+T(J_1X,J_1e_a,J_3e_a)\Big)\\
+\frac12\sum_{a=1}^{4n}\Big(T(X,J_3e_a,e_a)-T(J_1X,J_3e_a,J_2J_3e_a)\Big)=0.
\end{multline}
Now, \eqref{traces} follow from \eqref{2}, \eqref{3}, \eqref{31} and \eqref{32} which completes the proof of Lemma~\ref{1}.
\end{proof}
\subsection{Proof of Theorem~\ref{main2}}  We recall that the non-degenerate
2-form $\Omega_i=F_j+\sqrt{-1}F_k$ is of type (2,0) with respect
to the complex structure $J_i$ and it is parallel with respect to
the HKT-connection, $\nabla\Omega_i=0$. Hence, the (2n,0)-form
$\Omega_i^{n}$ is non-degenerate complex volume form which is
$\nabla$-parallel, $\nabla\Omega_i^{n}=0$.

Let the Lee form of the HKT-structure be an exact form,
$\theta=df$.  We claim that the (2n,0)-form
$\Phi=e^{-2f}\Omega_i^{n}$ is parallel with respect to the Obata
connection.

Let $e_1,\dots,e_{2n},J_ie_1,\dots,J_ie_{2n}$ be an orthonormal
basis  of $TM$ and $E_{\alpha}=e_{\alpha}-\sqrt{-1}J_ie_{\alpha},
\quad \alpha=1,\dots,2n$ be a basis of the (1,0)-space
$T^{1,0}_{J_i}M$ with respect to $J_i$. The complex conjugate of
$E_{\alpha}$ is as usual
$\overline{E_{\alpha}}=e_{\alpha}+\sqrt{-1}J_ie_{\alpha}$.

For a real vector $X$, we calculate from \eqref{conn} that
\begin{multline}\label{6ob}
(\Ob_X\Phi)(E_1,\dots,E_{2n})=(\nabla_X\Phi)(E_1,\dots,E_{2n})-
\sum_{\alpha=1}^{2n}A(X,E_{\alpha},\overline{E_{\alpha}})\Phi(E_1,\dots,E_{2n})\\=-2df(X)\Phi(E_1,\dots,E_{2n})+
e^{-2f}(\nabla_X\Omega_i^{n})(E_1,\dots,E_{2n})+2\theta(X)\Phi(E_1,\dots,E_{2n})=0,
\end{multline}
since $\nabla\Omega_i^{n}=0, \quad \theta=df$ and the  identity
$\sum_{\alpha=1}^{2n}A(X,E_{\alpha},\overline{E_{\alpha}})=-2\theta(X)$.
To see the latter, we calculate using Lemma~\ref{1} that
\begin{multline}\label{7ob}
\sum_{\alpha=1}^{2n}A(X,E_{\alpha},\overline{E_{\alpha}})\\
=\sum_{a=1}^{2n}\Big(A(X,e_a,e_a)+A(X,J_ie_a,J_ie_a)+\sqrt{-1}\Big[A(X,e_a,J_ie_a)-A(X,J_ie_a,e_a)\Big]\Big)\\
=-2\theta(X).
\end{multline}
For the converse, suppose that there exist a (2n,0)-form $\Psi$
which  is parallel with respect to the Obata connection,
$\Ob\Psi=0$. Hence, $|\Psi|^2>0$. We have similarly as above that
\begin{equation}\label{8ob}
0=(\Ob_X\Psi)(E_1,\dots,E_{2n})=(\nabla_X\Psi)(E_1,\dots,E_{2n})-
\sum_{\alpha=1}^{2n}A(X,E_{\alpha},\overline{E_{\alpha}})\Psi(E_1,\dots,E_{2n}).
\end{equation}
Apply \eqref{7ob} to \eqref{8ob} to conclude
\begin{equation}\label{9ob}
(\nabla_X\Psi)(E_1,\dots,E_{2n})=-2\theta(X)\Psi(E_1,\dots,E_{2n}).
\end{equation}
The identity \eqref{9ob} yields
$$\theta=-\frac14d(ln|\Psi|^2)$$
since the HKT-connection preserves the hyperhermitian structure.

Thus, the proof of Theorem~\ref{main2} is completed.

\section{Curvature of the Obata connection}

Let $(M,g,H,\nabla)$ be a 4n-dimensional HKT manifold. Let
$R=[\nabla,\nabla]-\nabla_{[\ ,\ ]}$ be the curvature tensor of
$\nabla$ and $R^{ob},R^g$ be the curvature of the Obata and the
Levi-Civita connection, respectively. Further  we used the
superscript $.^{ob}$, (resp. $.^g$) to denote tensors obtained
from the Obata connection $\Ob$ (resp. obtained from the
Levi-Civita connection $\nabla^g$). We denote the curvature tensor
of type (0,4) by the same letter, $R(X,Y,Z,U):=g(R(X,Y)Z,U)$. Note
that $\Ro$ is not skew-symmetric with respect to the second pair
of arguments since $\Ob$ is not a metric connection. The Ricci
tensor $Ric$, the Ricc-type 2-forms $\rho,\rho_s$ and the scalar
curvatures $Scal, Scal_s$ are defined as follows
\begin{equation*}
\begin{aligned}
&Ric(X,Y) =\sum_{a=1}^{4n}R(e_a,X,Y,e_a)\quad Scal=\sum_{a=1}^{4n}Ric(e_a,e_a), \quad Scal_s=\sum_{a=1}^{4n}Ric(J_se_a,e_a),\\
&\rho(X,Y)=\sum_{a=1}^{4n}R(X,Y,e_a,e_a),\quad \rho_s(X,Y)=\frac12\sum_{a=1}^{4n}R(X,Y,e_a,I_se_a), \quad
s=1,2,3.
\end{aligned}
\end{equation*}

We have

\begin{prop}\label{riobp}
On a hypercomplex manifold $(M,H)$, for $s=1,2,3$,  we have
\begin{equation}\label{riob}
Ric^{ob}(J_sX,J_sY)+Ric^{ob}(Y,X)=2\rho^{ob}_s(J_sX,Y), \quad Ric^{ob}(X,Y)-Ric^{ob}(Y,X)=-\rho^{ob}(X,Y).
\end{equation}
On an HKT manifold we have:
\begin{itemize}
\item[a)] The exterior derivative of the Lee form of an HKT manifold is an (1,1) form with respect to the hypercomplex structure,
$$d\theta(J_sX,J_sY)=d\theta(X,Y),\quad s=1,2,3;
$$
\item[b)] The Ricci tensor of the Obata connection of an HKT manifold is skew-symmetric determined by the Lee form and we have the identities
\begin{equation}\label{riobhkt}
Ric^{ob}(X,Y)=d\theta(X,Y), \quad \rho^{ob}=-2d\theta, \quad
\rho^{ob}_s=0, \quad s=1,2,3.
\end{equation}
In particular, the Ricci tensor of the Obata connection of an HKT manifold is an (1,1) form with respect to the hypercomplex structure, $Ric^{ob}(J_sX,J_sY)=Ric^{ob}(X,Y).$
\item[c)] The scalar curvatures of the Obata connection vanish,
$$Scal^{ob}=Scal^{ob}_s=0, \quad s=1,2,3.$$
\end{itemize}
\end{prop}
\begin{proof}
Let $g$ be a Riemannian metric hermitian compatible with the
hypercomplex structure.  Such a metric always exists. For example, take
any Riemannian metric $h$ then  the metric
$g(X,Y)=h(X,Y)+\sum_{s=1}^3h(J_sX,J_sY)$ is hyperhermitian. The
first Bianchi identity and the conditions $\Ro J_s=J_s\Ro, \quad
s=1,2,3$ imply the following sequence of identities
\begin{equation*}
\begin{aligned}
Ric^{ob}(X,Y)=-\sum_{a=1}^{4n}\Big(\Ro(X,Y,e_a,e_a)+\Ro(Y,e_a,X,e_a)\Big)=-\rho^{ob}(X,Y)+Ric^{ob}(Y,X);\\
2\rho_s^{ob}(X,Y)=\sum_{a=1}^{4n}\Ro(X,Y,e_a,J_se_a)=\sum_{a=1}^{4n}\Big(\Ro(Y,e_a,J_sX,e_a)+\Ro(e_a,X,J_sY,e_a)\Big)
\\=-Ric^{ob}(Y,J_sX)+Ric^{ob}(X,J_sY), \quad s=1,2,3
\end{aligned}
\end{equation*}
which proves \eqref{riob}.

Now, let $(M,g,H,\nabla)$ be an HKT manifold.
Using \eqref{conn} we obtain after standard calculations that the curvature of the Obata connection and the HKT-connection are related by
\begin{multline}\label{curvoh}
\Ro(X,Y,Z,U)=R(X,Y,Z,U)+(\nabla_XA)(Y,Z,U)-(\nabla_YA)(X,Z,U)\\+A(T(X,Y),Z,U)+A(X,A(Y,Z),U)-A(Y,A(X,Z),U),
\end{multline}
where $A$ is given by the second equation in \eqref{conn}.

Further, since the HKT-connection preserves the hyperhermitian structure its holonomy is contained in $Sp(n)$ and we have
\begin{equation}\label{4}
\rho=\rho_s=0, \qquad s=1,2,3.
\end{equation}
Taking the traces in \eqref{curvoh} and using  \eqref{traces} and \eqref{4}, we obtain
\begin{multline}\label{5}
\rho^{ob}(X,Y)= -2(\nabla_X\theta)Y+2(\nabla_Y\theta)X - 2\theta(T(X,Y))\\
+\sum_{a,b=1}^{4n}\Big(A(X,e_b,e_a)A(Y,e_a,e_b)-A(Y,e_b,e_a)A(X,e_a,e_b)\Big)=-2d\theta(X,Y),
\end{multline}
where we used the expression of the exterior derivative of an one
form $\alpha$ with  respect to a metric connection with torsion
$T$, $d\alpha(X,Y)=(\nabla_X\alpha)Y-(\nabla_Y\alpha)X
+\alpha(T(X,Y))$.

We calculate from \eqref{curvoh}, using  \eqref{traces} and \eqref{4} that
\begin{multline}\label{6}
\rho^{ob}_s(X,Y)=\sum_{a,b=1}^{4n}\Big(A(X,e_b,J_se_a)A(Y,e_a,e_b)-A(Y,e_b,J_se_a)A(X,e_a,e_b)\Big)\\
=\sum_{a,b=1}^{4n}\Big(A(Y,e_a,J_se_b)\Big[A(X,J_se_b,J_se_a)-A(X,e_b,e_a)\Big]\Big)=0,
\quad s=1,2,3.
\end{multline}
The last equality in \eqref{6} follows from the identity
$A(X,J_sY,J_sZ)-A(X,Y,Z)=0, \quad s=1,2,3$ which is a  consequence
from the fact that both the Obata and the HKT-connections preserve
the hyperhermitian structure.

The second and the third equality in \eqref{riobhkt} follow from \eqref{5} and \eqref{6}.

Using \eqref{6}, we obtain from \eqref{riob} that
\begin{equation}\label{riccom}
\begin{aligned}
&Ric^{ob}(J_sX,J_sY)+Ric^{ob}(Y,X)=0,  s=1,2,3 \quad {\rm
yielding}\\
&Ric^{ob}(J_sX,J_sY)=Ric^{ob}(J_tX,J_tY)=Ric^{ob}(X,Y), \quad
s,t=1,2,3.
\end{aligned}
\end{equation}
The two equalities \eqref{riccom}  lead to
$Ric^{ob}(X,Y)+Ric^{ob}(Y,X)=0$ which combined with the second
equality in \eqref{riob} and \eqref{5} imply the first equality in
\eqref{riobhkt}. This proves b).

The condition a) and $Scal^{ob}=0$ follow from  the second
equality in \eqref{riccom} and the first equality in
\eqref{riobhkt}.

To complete the proof of c) we have to show that $d\theta$ is
completely trace-free.  Fix $s\in \{1,2,3\}$ and consider the
1-form $J_s\theta$ defined by $J_s\theta(X)=-\theta(J_sX)$. The
condition $\theta=\delta F_s\circ J_s$ implies $J_s=\delta F_s$
and in particular the 1-form $J_s\theta$ is co-closed,
$\delta(J_s\theta)=0$. Expressing the latter in terms of
$\nabla^g$ and $\nabla$, we get
$$0=\delta(J_s\theta)=-\sum_{a=1}^{4n}(\nabla^g_{e_a}J_s\theta)(e_a)=-\sum_{a=1}^{4n}(\nabla_{e_a}J_s\theta)(e_a)
=\sum_{a=1}^{4n}(\nabla_{e_a}\theta)(J_se_a)$$
where  the third equality follows from $\nabla^g=\nabla-\frac12T$ and the fourth equality is a consequence of $\nabla J_s=0$. Then we have
$$\sum_{a=1}^{4n}d\theta(e_a,J_se_a)=\sum_{a=1}^{4n}\Big[2(\nabla_{e_a}\theta)(J_se_a)+\theta(T(e_a,J_se_a))\Big]=g(\theta,J_s\theta)=0$$
where we used  the expression of $d\theta$ in terms of the torsion
connection $\nabla$ and the definition of the Lee form
\eqref{leet}.
\end{proof}
It is known from \cite{AM} that the restricted holonomy group of
the Obata connection on an hypercomplex manifold is a subgroup of
$SL(n,\mathbb H)$ if and only if its Ricci tensor vanishes,
$Ric^{ob}=0$. On the other hand we have the inclusions
$SL(n,\mathbb H)\subset SL(2n,\mathbb C)\subset SL(4n,\mathbb R)$
which shows that $Hol(\Ob)\subset SL(n,\mathbb H)$ exactly when
all Ricci two forms of the Obata connection vanish. We obtain from
Proposition~\ref{riobp}  the next
\begin{cor}\label{riob2}
The restricted holonomy group of the Obata connection  on an HKT
manifold is contained in $SL(n,\mathbb H)$ if and only if the Lee
form is closed, $d\theta=0$;
\end{cor}
Note that Corollary~\ref{riob2} also follows from
Theorem~\ref{main2}.

We remark that  if the Lee form is closed but not exact  the
restricted holonomy group of the Obata connection is contained in
$SL(n,\mathbb H)$ but the whole  holonomy group of the Obata
connection may not be contained in $SL(n,\mathbb H)$. As pointed
out in \cite{V08} this happens  in the case of Hopf manifolds
$(\mathbb H^n-\{0\})/\Gamma$ which have flat Obata connection but
do not admit a holomorphic volume form and therefore these HKT
manifolds are not $SL(n,\mathbb H)$ manifolds.

\subsection{Non existence of HKT metric} As a direct consequence of
Proposition~\ref{riobp} one gets a simple criterion for
non-existence of HKT metric in terms of the Ricci-type tensors of
the Obata connection. Comparing the statements  in
Proposition~\ref{riobp}, we obtain
\begin{cor}\label{non}
Let $(M,H)$ be a hypercomlex manifold. Then there is no HKT
structure on $M$ compatible with the hypercomplex structure $H$ if any of the following three conditions hold:
\begin{itemize}
\item[a)] The Ricci tensor of the
Obata  connection is either not skew-symmetric or not (1,1)-form
with respect to the hypercomplex structure;
\item[b)] At least one of the
Ricci-forms of the Obata connection does not vanish identically,
$\rho_s^{ob}\not=0$ for some $s\in \{1,2,3\}$;
\item[c)] At least one of the scalar curvatures of the Obata connection is different from zero.
\end{itemize}
\end{cor}
We remark that if  the conditions a), b) and c) of Corollary~\ref{non} are satisfied  then this does not imply the existence of HKT structure due to the compact examples presented by A. Swann
\cite{S} of $SL(n,\mathbb H)$ manifolds which have vanishing Ricci
tensor of the Obata connection \cite{AM} but do not admit any HKT
structure.

\section{HyperK\"ahler HKT spaces}

In this section we give sufficient conditions a compact HKT
manifold to be hyperK\"ahler in terms of the traces of the
exterior derivative of the torsion and the *-scalar curvature of
the Levi-Civita connection.

It was observed in \cite{AI}  that a strong balanced KT manifold
is K\"ahler (see \cite{FPS} for a different independent proof).
More general, the equality (2.13) in \cite{AI} written in the form (see \cite[(3.10)]{IP1}),
\begin{equation}\label{old}\sum_{a,b=1}^{2n}dT(e_a,Je_a,e_b,Je_b)=8\delta\theta+8|\theta|^2-\frac43|T|^2,
\end{equation}
shows that a balanced KT manifold satisfying $\sum_{a,b=1}^{2n}dT(e_a,Je_a,e_b,Je_b)=0$ is
K\"ahler. If the manifold is  compact, a particular case of the
vanishing theorem \cite[Theorem~4.1]{IP1}, \cite{IP} states
\begin{thrm}\cite{IP1}\label{kt}
A compact (non K\"ahler) KT manifold  with restricted holonomy of the
KT-connection contained in $SU(n)$ satisfying the condition
$\sum_{a,b=1}^{2n}dT(e_a,Je_a,e_b,Je_b)=0$  admits no holomorphic volume
form.
\end{thrm}
We recall the slightly general notion of a QKT manifold which is
defined as   a quaternionic hermitian manifold of dimension $4n>4$ admitting a linear
connection preserving the quaternionic structure and having
totally skew-symmetric torsion which is an (1,2)+(2,1) three form
with respect to the quaternionic structure, the notion
investigated by Howe, Opfermann and Papadopoulos \cite{HP2} in
connection with supersymmetric sigma models with Wess-Zumino term.
We note that an HKT manifold is always a QKT manifold. The
QKT-torsion 1-form $t$ defined in \cite{I1} coincides (up to a
sign) with the Lee form $\theta$ in case of an HKT space.

For a fixed $s\in\{1,2,3\}$, the scalar curvature $Scal^g_s$ of
the Levi-Civita connection $\nabla^g$ is also known as *-scalar
curvature considering $(M,g,J_s)$ as an almost hermitian manifold.
The *-scalar curvature is also equal to
$Scal^g_s=\sum_{a=1}^{4n}\rho^g_s(J_se_a,e_a)$ due to the first
Bianchi identity for $R^g$.

Since any HKT space is a QKT manifold, it follows from
\cite[Proposition~3.4, Proposition~3.1]{I3}  that the three
*-scalar curvatures  on an HKT manifold of dimension grater then four coincide,
$Scal^g_1=Scal^g_2=Scal^g_3$, and the common scalar curvature
$Scal^g_H=Scal^g_1$, called \emph{the *-scalar curvature of an HKT
manifold},  is given by
\begin{equation}\label{scal}
Scal^g_H=\frac18\sum_{a,b=1}^{4n}dT(e_a,J_1e_a,e_b,J_1e_b)+\frac1{12}|T|^2
\end{equation}
because the Ricci forms of the HKT connection $\nabla$ vanish,
$\rho_s=0, s=1,2,3.$ In fact, the proof of \cite[Proposition~3.4]{I3} shows that the above conclusions hold also in dimension four. Indeed, for any HKT manifold, the formula (3.12) in \cite{I3}, taken with $t=-\theta$, reads
$$\rho^g(X,J_sY)=\frac12(\nabla_X\theta)Y+\frac12(\nabla_{J_sY}\theta)J_sX-\frac12\theta(J_sT(X,J_sY))
+\frac14\sum_{a,b=1}^{4n}T(X,e_a,e_b)T(J_sY,J_se_a,e_b)$$
for $s=1,2,3$, where we used  $\rho_s=0$ for an HKT manifold. The trace of the above equality with an application of \cite[Lemma~3.2]{I3} and \eqref{leet} gives
\begin{equation}\label{scal1}
Scal^g_s=\delta\theta+|\theta|^2-\frac1{12}|T|^2=\frac18\sum_{a,b=1}^{4n}dT(e_a,J_se_a,e_b,J_se_b)+\frac1{12}|T|^2,\quad s=1,2,3,
\end{equation}
where we applied \eqref{old} to obtain the second  equality of \eqref{scal1}.

We derive from Theorem~\ref{main2} and the vanishing results in
\cite{IP1,IP} the next
\begin{thrm}\label{qsvan}
A compact  HKT manifold with an exact Lee form  is hyperK\"ahler
if  any of the following two conditions hold
\begin{itemize}
\item[1).] The function
$h=-\frac14\sum_{a,b=1}^{4n}dT(e_a,J_1e_a,e_b,J_1e_b)$ vanishes
identically, $h=0$; \item[2).] the *-scalar curvature is zero,
$Scal^g_H=0$.
\end{itemize}
\end{thrm}
\begin{proof}
We apply the vanishing results from \cite{IP,IP1} to show that a compact non hyperK\"ahler HKT manifold satisfying any of the conditions 1), 2) of the theorem admits no holomorphic volume form, a contradiction with Corollary~\ref{cotriv}.

Since the holonomy of the torsion connection $\nabla$ of an HKT manifold is contained in $Sp(n)\subset SU(2n)$ we apply
\cite[Corollary~4.2~(b)]{IP1}. Corollary~4.2~(b) in \cite{IP1} states that if the function $|C|^2-h$, where $C$ is the torsion of the Chern connection of a KT manifold $(M,g,J)$, is strictly positive then $p_m(J)=dim H^0(M,\mathcal O(K^m)$ vanish for $m>0$ and, in particular, the complex manifold $(M,J)$ admits no holomorphic volume form.

We recall that the torsion $C$ of the Chern connection of a KT manifold $(M,g,J)$
is expressed in terms of the torsion three form $T$ as follows, see e.g. \cite{IP1}
$$g(C(X,Y),Z)=\frac12T(X,JY,JZ)+\frac12T(JX,Y,JZ).$$ The last equality together with an application of
\cite[Lemma~3.2]{I3} implies
\begin{equation}\label{chern}
|C_1|^2=|C_2|^2=|C_3|^2=\frac13|T|^2,
\end{equation}
where $C_1,C_2,C_3$ denote the torsion of the Chern connections of $(M,g,J_1), (M,g,J_2), (M,g,J_3)$, respectively.

If $h=0$ then clearly $|C_s|^2-h=|C_s|^2=\frac13|T|^2, s=1,2,3,$ where we applied \eqref{chern}.

Further, the condition $Scal^g_H=0$ together with \eqref{scal} gives
$h=\frac16|T|^2$. Using \eqref{chern} we obtain
$|C_s|^2-h=\frac16|T|^2, s=1,2,3$.

Hence, in both cases of the conditions
of the theorem each of the functions $|C_s|^2-h, s=1,2,3$ is a positive multiple of
$|T|^2$ and therefore it is strictly positive if the HKT space is not
hyperK\"ahler. Now, \cite[Corollary~4.2~(b)]{IP1} shows that the non
hyperK\"ahler  HKT manifold admits no holomorphic volume form
which contradicts Corollary~\ref{cotriv} since the Lee form is exact.
\end{proof}


For an HKT manifold the exterior derivative $dT$ of the torsion
three form is of type $(2,2)$ with respect to the hypercomplex
structure $H$ and, as shown in \cite{I1}, we have the equalities
$\sum_{a=1}^{4n}dT(e_a,J_1e_a,X,J_1Y)=\sum_{a=1}^{4n}dT(e_a,J_2e_a,X,J_2Y)=\sum_{a=1}^{4n}dT(e_a,J_3e_a,X,J_3Y)$
which allows us to define \emph{an almost strong HKT manifold} as
an HKT manifold satisfying the condition
$\sum_{a=1}^{4n}dT(e_a,J_1e_a,X,J_1Y)=0$. Theorem~\ref{qsvan} yields the following
\begin{cor}
A compact almost strong HKT manifold with an exact Lee form  is hyperK\"ahler.
\end{cor}

\end{document}